\documentclass[12pt]{amsart}
\usepackage{epsfig,color}

\makeatother

\theoremstyle{definition}

\def\fnum{equation} 
\newtheorem{Thm}[\fnum]{Theorem}
\newtheorem{Cor}[\fnum]{Corollary}

\newtheorem{Lem}[\fnum]{Lemma}
\newtheorem{Con}[\fnum]{Conjecture}

\newtheorem{Pro}[\fnum]{Proposition}

\numberwithin{equation}{section}

\newcommand{\nn}{{\bf{n}}}

\newcommand{\dist}{{\text {dist}}}

\newcommand{\Hess}{{\text {Hess}}}

\def\RR{{\bold R}}

\def\SS{{\bold S}}

\newcommand{\dv}{{\text {div}}}
\newcommand{\e}{{\text {e}}}

\newcommand{\cC}{{\mathcal{C}}}

\newcommand{\cHP}{{\mathcal{HP}}}

\newcommand{\cH}{{\mathcal{H}}}

\newcommand{\cP}{{\mathcal{P}}}
\newcommand{\cS}{{\mathcal{S}}}

\newcommand{\eqr}[1]{(\ref{#1})}

\title[Differentiability of the arrival time]{Differentiability of the arrival time}
\author[]{Tobias Holck Colding}%
\address{MIT, Dept. of Math.\\
77 Massachusetts Avenue, Cambridge, MA 02139-4307.}
\author[]{William P. Minicozzi II}%
 
\thanks{The  authors
were partially supported by NSF Grants DMS  11040934 and DMS 1206827.}

\email{colding@math.mit.edu  and minicozz@math.mit.edu}

\begin{document}

\maketitle

\begin{abstract}
 For a monotonically advancing front (cf. \cite{Se2}), the arrival time is the time when the front reaches a given point.
We show that it is twice differentiable everywhere with uniformly bounded second derivative.  It is smooth away from the critical points where the equation is degenerate.  We also show that the critical set has finite codimensional two Hausdorff measure.

For a monotonically advancing front, the arrival time is equivalent to the level set method; a priori not even differentiable but only satisfies the equation in the viscosity sense, \cite{ChGG},  \cite{ES1}, \cite{OsSe}.  
Using that it is twice differentiable and that we can identify the Hessian at critical points,
 we show that it satisfies the equation in the classical sense.  

The arrival time has a game theoretic interpretation.
For the linear heat equation, there is a game theoretic interpretation that relates to 
  Black-Scholes option pricing.

From variations of the Sard and \L{}ojasiewicz theorems, we relate  differentiability to whether or not singularities all occur at only finitely many times for flows.  
\end{abstract}

\section{Introduction}

Let $M_t\subset \RR^{n+1}$ be a mean curvature flow (MCF) starting at a closed smooth mean convex hypersurface $M_0$.   Under the flow, $M_t$ remains mean convex and, thus, moves monotonically  inward as it
sweeps out{\footnote{Each $M_t$ bounds a compact domain $\Omega_t$ and $\Omega_t = \cup_{s\geq t} M_s$; see Subsection \ref{ss:31}.}} the compact domain $\Omega_0$ bounded by $M_0$.   The arrival time $u:\Omega_0 \to \RR$ is the time  when the front $M_t$ arrives at a point $x \in \Omega_0$
\begin{align}
 u(x)=\{t\,|\,x\in M_t\}\, .
\end{align}
Even though $M_0$ is smooth, 
the later $M_t$'s need not be but are  given by the level set method and the function $u$ is a priori only continuous.




\vskip2mm
Our main result is the differentiability of the arrival time for mean convex flows:

\begin{Thm}	\label{t:mainarrival}
The arrival time is twice differentiable everywhere, smooth away from the critical set,
 and has uniformly bounded second derivative.  Moreover,  
\begin{itemize} 
\item The critical set
is contained in   finitely many compact embedded $(n-1)$-dimensional Lipschitz submanifolds plus a set of dimension at most $n-2$.{\footnote{In fact, we prove much more.
For example,
 in $\RR^3$, the critical set is contained in finitely many (compact) embedded Lipschitz curves where the Hessian have   eigenvalues $0$ (with multiplicity $1$) and $-1$ (with multiplicity $2$) together with a countable set where the Hessian is $-\frac{1}{2}\,\delta_{i,j}$.}}
\item At each critical point the hessian is symmetric and has only two eigenvalues $0$ and $-\frac{1}{k}$; $0$ has multiplicity $n-k$ (which could be $0$) and $-\frac{1}{k}$ has multiplicity $k+1$.  Here $k=1,\cdots, n$.  
\item It satisfies the equation everywhere in the classical sense.  
\end{itemize}
\end{Thm}

\vskip2mm
A key point  is that the second derivatives exist even at the critical set where the flow is singular.  
Though it is known to be Lipschitz by a result of Evans and Spruck, \cite{ES1}, a priori it is not even differentiable but only satisfies
 the second order degenerate equation in the viscosity sense, \cite{ChGG}, \cite{E1}, \cite{ES1}.   In the convex case, where   the flow is smooth except at the point it becomes extinct, Huisken showed that the arrival time is $C^2$ in all dimensions, \cite{H2}, \cite{H3}.  
There is even more regularity in the plane, where Kohn and Serfaty showed, using \cite{GH}, that it is at least $C^3$, \cite{KS1}.  For $n>1$, Sesum, \cite{S}, using \cite{H3},  showed that  third derivatives do not exist even for convex hypersurfaces.   However, we will see in Theorem \ref{t:axidi} that there is more regularity in the direction of the zero eigenvalues.

\subsection{Level set method and viscosity solutions}   \label{s:levelset}
The idea behind the level set method is given an interface $M_0$ in $\RR^{n+1}$ of codimension one, bounding a (perhaps multiply connected) closed region $\Omega_0$, one  analyzes and computes its subsequent motion under a velocity field, \cite{OJK}, \cite{OsFe}, \cite{OsSe},  \cite{Se1}, \cite{Se3}. 
This velocity can depend on position, time, the geometry of the interface (e.g. its normal or its mean curvature) and the external physics. The idea is to define an at least continuous function $v(x, t)$, that represents the interface as the set where $v(x, t) = 0$. 
The level set function $v$ has the following properties{\footnote{Strictly speaking, the interface could have an interior; this does not occur in the case considered here.}}
\begin{itemize}
\item $v(x,t) > 0$ for $x\in \Omega_t \setminus \partial \Omega_t$.
\item $v(x,t) < 0$ for $x\in \RR^{n+1}\setminus \Omega_t$. 
\item $v(x,t) = 0$ for $x \in \partial \Omega_t=M_t$.
\end{itemize}
This is of great significance for numerical computation and in applications, primarily because topological changes such as breaking and merging are well defined.
When all level sets evolve by MCF,  this motion becomes
\begin{align} 
v_t=|\nabla v|\,\dv \left(\frac{\nabla v}{|\nabla v|}\right)\, ,
\end{align}
and has been studied extensively.  Whereas the work of Osher and Sethian was numerical, Evans and Spruck, \cite{ES1}, and, independently, Chen, Giga, and Goto, \cite{ChGG} provided the theoretical justification for this approach.  This is analytically subtle, principally because the mean curvature evolution equation is nonlinear, degenerate, and indeed even undefined at points where $\nabla v = 0$.  Moreover, $v$ is a priori not even differentiable, let alone twice differentiable.   They resolved these problems by introducing an appropriate definition of a weak solution, inspired by the notion of   ``viscosity solutions"; \cite{Ca}, \cite{CrIL}.  

When the front is advancing monotonically (so the mean curvature is non-negative) Evans and Spruck, \cite{ES1}, showed that $v(x,t)=u(x)-t$, where $u$ is Lipschitz and satisfies (in the viscosity sense; see Section \ref{s:classical} for details) 
\begin{align} \label{e:levelset}
-1=|\nabla u|\,\dv \left(\frac{\nabla u}{|\nabla u|}\right)\, .
\end{align}
Obviously, $u^{-1}(t)$ is $M_t$; in other words: $u$ is the arrival time; see  \cite{Se2} for numerics.

\subsection{Game theoretic intepretation}  We will next briefly explain a game theoretical interpretation of the arrival time following, essentially verbatum, Kohn and Serfaty, \cite{KS2}, see also \cite{E2}, \cite{GL}, \cite{K}, \cite{KS1}--\cite{KS3}, \cite{Sp}. 

In a two-person game, with players Paul and Carol and a small parameter $\epsilon$. Paul is initially at some point $x$ in a bounded domain $\Omega$; his goal is to exit as soon as possible. Carol wants to delay his exit as long as possible. The game proceeds as follows:

\begin{itemize}
\item Paul chooses a direction, i.e., a unit vector $|v|=1$.
\item Carol can either accept or reverse Paul's choice, i.e., she chooses $b=\pm 1$.
\item Paul then moves   $2\,\epsilon$ in the possibly reversed direction, i.e., from $x$ to $x + 2\,\epsilon\, b\, v$. 
 
\end{itemize}
This cycle repeats until Paul reaches $\partial \Omega$. Let $u_{\epsilon}$ be $\epsilon^2$ times the number of steps that Paul needs to exit.  
Can Paul exit? Yes indeed.  For Paul's optimal strategy the time before exit $\lim_{\epsilon\to 0}u_{\epsilon}$ is exactly the arrival time.\footnote{\cite{KS1}  proved that $\lim_{\epsilon\to 0}u_{\epsilon}$ exists and is the arrival time $u$.  They gave an error estimate if $u$ is $C^3$.}

The Paul-Carol game was introduced in the 1970s by Joel Spencer, \cite{Sp}, as a heuristic for the study of certain combinatorial problems.  See \cite{KS2}, \cite{KS3} for other  game theoretic interpretations of parabolic equations, including one for the linear equation that is 
close to Black-Scholes option pricing.

\subsection{Finitely many singular times}
 
Differentiability is connected to the following well-known conjecture; \cite{AAG}, page 533 of \cite{W1},  conjectures 0.5 in \cite{CM4} and  6.5 in \cite{CMP}:\footnote{It is interesting to compare with a series of papers by Bamler, \cite{Ba}, on
 the long-time behavior of Ricci flows with surgery. He shows that only finitely many surgeries occur  if the surgeries are performed correctly.}

\begin{Con}	\label{c:con}
 For $M_t$ as above, the evolving hypersurface is completely smooth except at  finitely many times.\footnote{In low dimensions it is known by \cite{CM4} that connected components of the singular set are contained in time-slices and almost all time slices are completely smooth.}  
\end{Con}

 The connection is that the regularity of a function $u$ controls the size of the set $\cC$ of critical values.   In dimension three,
if   $u: \Omega \subset \RR^3 \to \RR$ is at least $C^3$, then $\cC$ has measure zero by Sard's theorem, see, e.g., \cite{Y}.  Moreover, if $u$ is $C^k$ for some $k> 3$, then $\cC$ has
 dimension at most $\frac{3}{k} < 1$.  Similarly, in higher dimensions.  
The arrival time is only twice differentiable which   does not seem like enough  to get measure zero from Sard's theorem.  However,  using \cite{CM4},  we will see that
 the critical set is contained in a finite union of compact curves plus countably many points and, thus,   Sard's theorem on the line suggests that $\cC$  has dimension at most $\frac{1}{2}$.   This heuristic argument uses the structure  established in \cite{CM4},  where the  bound  $\frac{1}{2}$ for the dimension was rigorously shown by a different argument.

When the arrival time is real analytic, there can be only finitely many singular times.  To see this, 
recall the   gradient \L{}ojasiewicz inequality, \cite{L},   \cite{CM2}, \cite{CM3}:  
\begin{quote}
A function $u$ is said to satisfy the gradient \L{}ojasiewicz inequality if for every $p\in \Omega$, there is a possibly smaller neighborhood $W$ of $p$ and constants $\beta\in (0,1)$ and $C>0$ such that $|u(x)-u(p)|^{\beta}\leq C\,|\nabla_x u|$  for all $x\in W$.
\end{quote}
\L{}ojasiewicz proved this inequality for   analytic functions and, thus, a critical point  
 has a neighborhood with no other critical values.  Consequently, $\cC$ is finite if $\Omega$ is compact:

\begin{Cor}
There are only finitely many singular times if the arrival time is real analytic.  
\end{Cor}

For mean convex MCF the    \L{}ojasiewicz inequality would say (since  $H\,|\nabla u|=1$ by \eqr{e:Hequal}) that for   $p\in \Omega$, there is a  neighborhood $W$ of $p$ and constants $\beta\in (0,1)$,  $C>0$ such that $H(x)\,|t(x)-t(p)|^{\beta}\leq C$  for all $x\in W$.
Since $|A|^2\leq C\,H^2$  for a mean convex MCF  by the parabolic maximum principle
for   $C$ depending on $M_0$, the   \L{}ojasiewicz inequality would say
\begin{align}
|A(x)|\,|t(x)-t(p)|^{\beta}\leq C\, .
\end{align}
In other words, $|A|$ blows-up at a singular point at most like $|t-t_0|^{-\beta}$, where $t_0$ is the singular time.    By a standard  ODE comparison argument, it blows up at least like $|t-t_0|^{-\frac{1}{2}}$.  If $|A|$  blows-up at a rate of at most $|t-t_0|^{-\frac{1}{2}}$, then the singularity is said to be type one.


 \section{Preliminary estimates}

\subsection{Mean convex MCF}
 We have already seen that the level set method is equivalent to the arrival time function for a monotonically advancing front. Namely, $v(x,t)=u(x)-t$ by Subsection \ref{s:levelset}.  Since we will use the equation for $u$ (i.e., \eqr{e:levelset}) several times we will deduce it for completeness.
Suppose therefore that $M_t \subset \RR^{n+1}$ is a mean convex MCF and $u=t$.  As we will explain in Section \ref{s:classical} the equation for $u$ is initially interpreted in classical sense at points where $\nabla u\ne 0$ and in the barrier, or viscosity sense, everywhere else.  For the sake of deducing the equation assume we are at a point where $\nabla u\ne 0$ and $u$ is smooth, then
\begin{align}
\dv_{M_t}\,\left(\frac{\nabla u}{|\nabla u|}\right)\,\frac{\nabla u}{|\nabla u|}=-x_t\, .
\end{align}
Since $u(x(t))=t$, we have that $\langle \nabla u,x_t\rangle =1$, or, equivalently, 
\begin{align}
x_t=\frac{1}{|\nabla u|}\,\frac{\nabla u}{|\nabla u|}=\frac{\nabla u}{|\nabla u|^2}\, .
\end{align}
Putting this together gives
\begin{align}	\label{e:elliptic}
-1=|\nabla u|\,\dv_{M_t}\,\left(\frac{\nabla u}{|\nabla u|}\right)=|\nabla u|\,\dv\,\left(\frac{\nabla u}{|\nabla u|}\right)=\Delta\,u- \Hess_u \left(  \frac{\nabla u}{|\nabla u|} , \frac{\nabla u}{|\nabla u|}
	\right) =\Delta_1u\, ;
\end{align}
where $\Delta_1$ is the
$1$-Laplacian.{\footnote{It is more common to define the $1$-Laplacian as $\dv \left( \frac{ \nabla u }{|\nabla u|} \right)$ which differs by a factor of $|\nabla u|$;  cf. \cite{E2}.}}  The operator $\Delta_1$ is the trace of the Hessian over the $n$-dimensional subspace orthogonal to $\nabla u$.  It is nonlinear and degenerate elliptic.  The ordinary Laplacian $\Delta$ is the sum of $\Delta_1$ and the $\infty$-Laplacian $\Delta_{\infty}u = \Hess_u \left(  \frac{\nabla u}{|\nabla u|} , \frac{\nabla u}{|\nabla u|} \right)$.  
The mean curvature $H$ (with respect to the outward unit normal of $\Omega_t$) is
\begin{align}  \label{e:Hequal}
H= |\nabla u|^{-1}  \, .
\end{align}

As an example, consider the shrinking round cylinder $\SS^k\times \RR^{n-k}$ in $\RR^{n+1}$ with radius $\sqrt{-2kt}$ for $t<0$ so that at time $-1$ the radius is $\sqrt{2k}$ and it becomes extinct at time $0$.  In this case, $u=-\frac{1}{2k}\left( x_1^2 + \dots + x_{k+1}^2 \right)$ is analytic and the critical set is $\{0\}\times \RR^{n-k}$.

\subsection{$C^{1,1}$ estimates away from the singular set}


\begin{Lem}	\label{l:C11}
There is a constant $C_{1,1}$ depending on the initial hypersurface so that
\begin{align}
	\left| \Hess_u \right| \leq C_{1,1}	{\text{ on the regular set (the complement of the singular set for the flow)}}\, .
\end{align}
\end{Lem}

In contrast, the derivatives of the hessian are not uniformly bounded in general, \cite{S}.    
To prove the lemma, we need that, by a result of Haslhofer-Kleiner, \cite{HaK}{\footnote{This applies by Andrews, \cite{An}, \cite{AnLM}.}}, for each
integer $\ell \geq 0$,  there is a 
 $C$ (depending on $M_0$ and $\ell$)
 so that 
\begin{align}  	\label{t:HaK}
	  \left| \nabla^{\ell} A \right| \leq C \, H^{\ell +1} \, .
\end{align}


\begin{proof}[Proof of Lemma \ref{l:C11}]
We will consider three cases for the various components of the Hessian.

Since $H = |\nabla u|^{-1}$ and, given unit vectors $e_i$ and $e_j$ tangent to the level set, we have
\begin{align}	\label{e:hessureg}
	A(e_i , e_j) = \frac{ \Hess_u (e_i , e_j)}{|\nabla u|}  = H \, \Hess_u (e_i , e_j)  \, ,
\end{align}
it follows that  $ \left|  \Hess_u (e_i , e_j) \right| = \frac{ |A(e_i , e_j)|}{H} $
  is uniformly bounded by the case $\ell = 0$ of \eqr{t:HaK}.

 For the double normal direction, use that $|\nabla u| = H^{-1}$ and $x_t = \frac{\nabla u}{|\nabla u|^2} = H \, \nn$ so that
\begin{align}
	-\Hess_u (\nn , \nn) &= - \nabla_{\nn} |\nabla u| = - \frac{1}{H} \, \partial_t \, H^{-1} = \frac{ \partial_t H }{H^3} 
	= \frac{ \Delta H}{H^3} + \frac{|A|^2}{H^2} \, ,
\end{align}
where the last equality is Simons' equation for $H$.  We have already bounded the last term and the remaining $\frac{ \Delta H}{H^3}$ term is bounded
by the case $\ell =2$ of \eqr{t:HaK}.

Similarly, the hessian in a mixed tangential/normal direction is given by
\begin{align}
	\Hess_u (e_i , \nn) = \nabla_{e_i} |\nabla u| = \nabla_{e_i} H^{-1} = -\frac{ H_i }{H^2} \, .
\end{align}
This term is bounded by the case $\ell =1$ of \eqr{t:HaK}, completing the proof.
\end{proof}

\subsection{Level set flow}		\label{ss:31}

Up until now, we have worked solely on the regular set for the flow where the solution $u$ is smooth and $\nabla u \ne 0$.  
We will now work across the singularities.

We begin by recalling the properties of mean convex MCF starting from a smooth closed mean convex hypersurface:
\begin{enumerate}
\item There is a unique Lipschitz function $u$ giving a viscosity solution of the level set flow.  
\item The flow is non-fattening (i.e., the level sets have no interior), 
 each level set is the boundary of a compact mean convex set,    and $u$ is defined everywhere on the interior.
 \item  The
 level sets define a Brakke flow of integral varifolds.
 \item Each tangent flow at a singularity is a multiplicity one (generalized) cylinder.

\item On the regular set for the flow, the function $u$ is smooth and $\nabla u \ne 0$.
 
\end{enumerate}

 \vskip2mm
 Claim (1) is theorem $7.4$ in 
  \cite{ES1}; cf. \cite{ChGG}.  Claim (2) follows from theorems $3.1$ and $3.2$ and corollary $3.3$ in \cite{W2}.  Claim (3) is theorem $5.1$ in \cite{W2}; cf. 
  \cite{ES2}, 
 \cite{MS}.   For Claim (4), see
   theorem $3$ in \cite{W4} (cf. \cite{HS}).  
   The last claim is almost a tautology.  If the flow is smooth in a neighborhood, we can write it as a normal graph $(x,t) \to x + v(x,t) \, \nn (x)$ of a smooth function $v(x,t)$ defined on the level set with
  $v(x,0) = 0$.  The map  is smoothly invertible and   $u$ is  the $t$ component of the inverse.  Invertibility follows from the inverse function theorem since the
   differential of this map has full rank since $H$ does not vanish.  Finally, $\nabla u$ does not vanish on the regular part because
     $H = |\nabla u|^{-1}$.

\subsection{Structure of the critical set}
 
  Let $\Phi: (x,t)\to x$ be the forgetful map from $\RR^{n+1}\times \RR$ with the parabolic metric{\footnote{The parabolic distance $\dist_{\cP}$ on $\RR^{n+1} \times \RR$
is $\dist_{\cP} \left( (x,s) , (y,t) \right) = \max \{ |x-y| , |s-t|^{\frac{1}{2} }  \}$.  The parabolic Hausdorff measure $\cHP_k$ is the $k$-dimensional Hausdorff measure with respect to $\dist_{\cP}$;
see   \cite{CM4}.}} 
   to $ \RR^{n+1}$ with the Euclidean metric.
By \cite{CM4}, 
 the space-time singular set $\cS$ is contained in   finitely many compact embedded $(n-1)$-dimensional Lipschitz submanifolds plus a set of dimension at most $n-2$.
 The same holds for $\Phi (\cS)$ since  $\Phi$  is distance non-increasing.  The critical set is contained in $\Phi (\cS)$ by (5), giving the first part of Theorem \ref{t:mainarrival}.  
Thus, also
\begin{align}
\cH_{n-1}(\{x\in\RR^{n+1}\,|\,\nabla_xu=0\}) \leq \cH_{n-1}(\Phi(\cS)) \leq \cHP_{n-1} (\cS) 
<\infty\, .
\end{align}

\subsection{The singular set equals the critical set}

\begin{Cor}	\label{c:uc1}
  $u$ is $C^{1,1}$ and the singular set is equal to the critical set   $\{ \nabla u = 0 \}$.
\end{Cor}

We will need the following standard extension:

\begin{Lem}	\label{l:exte}
Suppose that $\Omega \subset \RR^{n+1}$ is a connected open set with smooth closure
  and $S \subset \Omega$ is closed and has codimension two.  If $f : \Omega \setminus S \to \RR$ is smooth with
$|f| + |\nabla f| \leq C$, then $f$ can be extended continuously to $\Omega$ and the extension is Lipschitz.
\end{Lem}

\begin{proof}
Given points $x , y \in \Omega \setminus S$, we can find a $C^1$ path $\gamma$ from $x$ to $y$ whose length is at most twice the distance (in $\Omega$) from $x$ to $y$ and that avoids $S$
since $S$ has codimension two.  Applying the fundamental theorem of calculus along this path, the bound on $\nabla f$ gives  
\begin{align}
	|f(x) - f(y)| \leq 2 \, C \, \dist_{\Omega} (x,y) \, .
\end{align}
Since $x$ and $y$ were arbitrary,   $f$ is uniformly Lipschitz on $\Omega$.  
  It follows that $f$ has a unique continuous extension to all of $\Omega$ and that this extension is Lipschitz with the same bound.
\end{proof}

\begin{proof}[Proof of Corollary \ref{c:uc1}]
We will first see that each partial derivative $u_k$ is continuous on the entire domain $\Omega$.  Let $S$ be the singular set, so that it is closed, by definition, and 
has finite codimension two measure by \cite{CM4}. 
The function $u_k$ is smooth on $\Omega \setminus S$ with $\nabla u_k$ uniformly bounded   by 
Lemma \ref{l:C11}.  Hence,  by Lemma \ref{l:exte}, $u_k$ is defined on all of $\Omega$ and is Lipschitz.

To see that $\nabla u$ vanishes on the singular set, note that each singularity has a cylindrical tangent flow and, thus, is given as a limit of regular points with $H = |\nabla u|^{-1}$ going to infinity.
Since $|\nabla u|$ is continuous, we must have $|\nabla u| = 0$ at the singular point.

\end{proof}

\section{The second derivatives at critical points}
 
 We now prove the existence of the second derivatives at a critical point and show that
   the   Hessian is given by the tangent flow at the singularity. This depends crucially  on the uniqueness of the tangent flow proven in
    \cite{CM2}; cf. \cite{CIM}.
    
 To keep the notation simple, we will assume that $0$ is a critical point of $u$ and   $u(0) = 0$.  By  \cite{CM2}, there is a unique tangent flow at $0$ which is a multiplicity one cylinder.  After a rotation,  the level set flow corresponding to the blow up cylinder is given by the function 
 \begin{align}
 	w = - \frac{1}{2k} \, \sum_{i=1}^{k+1} x_i^2 \, ,
 \end{align}
and $\nabla w = - \frac{x_+}{k}$, where $x_+ = (x_1 , \dots , x_{k+1} , 0 , \dots , 0)$ is the projection of $x$ onto $\RR^{k+1}$.

 \begin{Pro}	\label{p:oneA}
The function $u$ is twice differentiable at $0$ and   $\Hess_u = \Hess_w$ at $0$.
\end{Pro}

To understand this,  rescale the flow by a factor $\alpha > 0$ so that the new arrival time becomes
\begin{align}
	u_{\alpha} (x) = \alpha^{-2} \, u (\alpha x ) \, .
\end{align}
This rescaling preserves the Hessian at $0$ (assuming it exists) and, as $\alpha$ goes to $0$, the rescaling converges (in a sense to be made precise) to the tangent flow at $0$.

\vskip2mm
We will prove that the difference quotients of $\nabla u$ have a limit.  We will consider two cases, depending on how the points approach $0$.  The main case is   the following:

\begin{Lem}	\label{l:maincase}
If $x_i$ is any sequence in $\RR^{n+1}$ with $u(x_i) < 0$, $x_i \to 0$, $ \frac{x_i}{|x_i|} \to  v$,  and
\begin{align}	\label{e:coneregion}
	\limsup_{i \to \infty} \, \frac{|x_i|^2}{|u(x_i)|} < \infty \, , 
\end{align}
then $\lim_{i \to \infty} \frac{\nabla u (x_i)}{|x_i|}$ exists and is equal to $\partial_{v} \, \nabla w$.
\end{Lem}
 
 \begin{proof}
 Since $u(x_i) < 0$ and $x_i \to 0$,   \cite{CM2} implies that the rescaled level sets
 \begin{align}	
 	\frac{ \{ u = u(x_i) \} }{\sqrt{-u(x_i)}} \,  
 \end{align}
 converge smoothly on compact sets to the cylinder $\Sigma_k = \{ x_1^2 + \dots + x_{k+1}^2 = 2k \} \subset \RR^{n+1}$.
 The condition \eqr{e:coneregion} implies that the points
 \begin{align}	\label{e:yiy}
 	y_i \equiv \frac{x_i}{\sqrt{-u(x_i)}}
 \end{align}
 lie in a bounded set. Choose a subsequence (still denoted by $y_i$) that converges to a limiting point $y \in \RR^{n+1}$.  Write $y$ as $y = (y_+ , y_-)$, where $y_+ \in \SS^{k}_{\sqrt{2k}}$ is the projection of $y$ on the first $k+1$ coordinates and $y_- \in \RR^{n-k}$ is the orthogonal part.   At $y_i$, the normal to the rescaled level set   and the mean curvature  are given by
 \begin{align}
 	\frac{\nabla u (x_i)}{|\nabla u (x_i)|} \,  {\text{ and }} \,  \frac{\sqrt{-u(x_i)}}{  |\nabla u (x_i)| } \, .
 \end{align}
Since these are converging  to the same quantities for the limit $\Sigma_k$ at $y$, we see that
\begin{align}
	\frac{\nabla u (x_i)}{|\nabla u (x_i)|} \to  \frac{-y_+}{\sqrt{2k} } \, {\text{ and }} \,
	\frac{\sqrt{-u(x_i)}}{  |\nabla u (x_i)| } \to \frac{\sqrt{k}}{\sqrt{2}}  
	\, .
\end{align}
Combining \eqr{e:yiy} with the two previous limits, we see that
\begin{align}	\label{e:diffq1}
	 \frac{\nabla u (x_i)}{|x_i|} = \frac{\nabla u (x_i)}{|\nabla u (x_i)|} \, \left(  \frac{\sqrt{-u(x_i)}}{  |\nabla u (x_i)| } \right)^{-1} \, |y_i|^{-1}
	 \to \frac{-y_+}{k \, |y|} = \frac{-v_+}{k} \, , 
\end{align}
where the last equality used that  $ \frac{x_i}{|x_i|} \to  v$ and $v_+$ is the projection of $v$ onto $\RR^{k+1}$.

This limit is independent of the choice of subsequence and, thus, the limit of the difference quotients exists and is given by 
\eqr{e:diffq1}.  Finally, we observe that this is equal to $\partial_{v} \, \nabla w$.
 
 \end{proof}

The previous lemma gives the Hessian, but requires that we approach $0$ in  a direction where $u< 0$.  By the next lemma, this covers every direction except along the axis. 

\begin{Lem}	\label{l:approach}
Given   $\epsilon > 0$, there exist $C , \delta > 0$ so that if $x = (x_+ , x_-) \in \RR^{k+1} \times \RR^{n-k}$ has
\begin{align}
	|x_+| \geq \epsilon \, |x_-| {\text{ and }} |x| \leq \delta \, , 
\end{align}
then $|x|^2 \leq -   C\, u (x)$.
\end{Lem}

\begin{proof}
This follows  from the uniqueness of \cite{CM2}.
\end{proof}
 
\begin{proof}[Proof of Proposition \ref{p:oneA}]
Fix a unit vector $v = (v_+ , v_-)$ and let $\delta_i > 0$ be a sequence converging to $0$. 
We will consider two cases.

Suppose first that 
  $|v_+| > 0$.   Lemma \ref{l:approach} implies that the sequence $x_i \equiv \delta_ i \, v$ (for $i$ large enough) satisfies the hypotheses of 
 Lemma \ref{l:maincase}.  Therefore, we conclude that 
 \begin{align}
 	\lim_{i \to \infty} \, \frac{ \nabla u(\delta_i \, v) }{\delta_i} = \frac{ - v_+ }{k} = \partial_v \, \nabla w \, .
 \end{align}
 
 Suppose next that $v_+ = 0$.  Given some  small $\epsilon > 0$,
 Lemma \ref{l:approach} implies that the sequence 
 \begin{align}
 	x_i \equiv   \delta_ i \, v + (\epsilon \, \delta_i , 0 , \dots , 0) 
\end{align}
 satisfies the hypotheses of 
 Lemma \ref{l:maincase} (for $i$ large enough).
 Therefore, we conclude that 
 \begin{align}	\label{e:secc}
 	\lim_{i \to \infty} \, \frac{ \nabla u(x_i) }{\delta_i} = \frac{ - (\epsilon , 0 , \dots , 0) }{k \sqrt{1+\epsilon}}   \, .
 \end{align}
  To relate this to the difference quotients in the direction $v$,  use the   $C^{1,1}$ bound  to get
  \begin{align}	\label{e:secd}
  	\left| \frac{\nabla u (x_i)}{\delta_i} -  \frac{\nabla u (\delta_i \, v ) }{\delta_i} \right| \leq \frac{  \left| \nabla u (x_i) - \nabla u (\delta_i \, v ) \right|}{\delta_i} \leq C \, \epsilon  \,  .
  \end{align}
Finally, since $\epsilon > 0$ is arbitrary, combining \eqr{e:secc} and \eqr{e:secd} gives
  \begin{align}
  	\lim_{i \to \infty} \, \frac{ \nabla u(\delta_i \, v ) }{\delta_i} = 0 = \partial_v \nabla w \, .
  \end{align}
\end{proof}

\section{Solving the equation classically}		\label{s:classical}
 
 Evans and Spruck (section $7.3$ in \cite{ES1}; cf. \cite{ChGG}) constructed a continuous 
 viscosity solution $u$ of \eqr{e:elliptic}  and showed   that it is   unique   and Lipschitz.
	Recall that $u$ is a viscosity solution if it is both
  a sub and super solution.  A  continuous function $u$ is a sub solution (super solutions are defined similarly) provided that:
	\begin{quote}
		If $\phi$ is a smooth function so that $u-\phi$ has a local maximum at $x_0$, then 
		\begin{enumerate}
		\item $\Delta_1 \, \phi \equiv \Delta\,\phi- \Hess_{\phi} \left(  \frac{\nabla \phi }{|\nabla \phi|} , \frac{\nabla \phi}{|\nabla \phi|}
	\right)  \geq -1$ at $x_0$ if $\nabla \phi (x_0) \ne 0$.
		\item $\Delta\,\phi- \Hess_{\phi} \left(  v , v
	\right)  \geq -1$ at $x_0$ for some vector $|v| \leq 1$ if $\nabla \phi (x_0)= 0$.
		\end{enumerate}
	\end{quote}
 
  We will say that a twice differentiable function $u$ is a classical solution of \eqr{e:elliptic} if:   
\begin{enumerate}
\item[(A)]  $\Delta_1 \, u \equiv \Delta\,u- \Hess_{u} \left(  \frac{\nabla u}{|\nabla u|} , \frac{\nabla u}{|\nabla u|}
	\right)  = -1$  where  $\nabla u \ne 0$.
		\item[(B)] $\Delta\,u- \Hess_{u} \left(  v , v
	\right)  = -1$ at $x_0$ for some vector $|v| = 1$ if $\nabla u (x_0)= 0$.
\end{enumerate}

\begin{Lem}	\label{l:consist}
If $w$   solves  \eqr{e:elliptic} classically, then $w$ is also a viscosity solution.
\end{Lem}

\begin{proof}
We must show that $w$ is both a sub and super solution.  To see that $w$ is a sub solution, suppose   $\phi$ is smooth and $w-\phi$ has a local maximum at an interior point $x_0$,
so  
\begin{align}	\label{e:firdt}
	\nabla w (x_0) = \nabla \phi (x_0) \, .
\end{align}
The second derivative test for twice differentiable functions ($11$ on page $115$ of \cite{R}) gives
\begin{align}	\label{e:2dt}
	\Hess_w (v ,v) \leq \Hess_{\phi}(v,v) {\text{ at }} x_0 {\text{ for any vector }} v \in \RR^{n+1} \, .
\end{align}
If $\nabla \phi \ne 0$, then \eqr{e:firdt} implies that we are taking traces over the same $n$-dimensional subspaces in (1) and (A); the inequality \eqr{e:2dt} then gives that 
(A) implies (1).
In the other case where $\nabla \phi = 0$, then let $v$ be the same unit vector in (1) as in (B).  Taking the trace over $v^{\perp}$, we see that (B) and \eqr{e:2dt}  imply (2).  We conclude that $w$ is a sub solution.
The proof that $w$ is a super solution follows similarly.
\end{proof}

The next lemma shows that spheres and cylinders give classical solutions to  \eqr{e:elliptic}.

\begin{Lem}	\label{l:spherecyl}
If $ \nabla u(0)=0$ and $\Hess_u$ exists at $0$ and equals $\Hess_w$ where 
  $w = - \frac{1}{2k} \, \sum_{i=1}^{k+1} x_i^2$, then $u$  is a classical solution to \eqr{e:elliptic} at $0$.
\end{Lem}

\begin{proof}
If $k=n$ (the spherical case), then $\Hess_u (0)$ is diagonal with all $n+1$ eigenvalues equal to $- \frac{1}{n}$.  In this case, (B) holds at $0$ for any unit vector $v$.

When $k< n$, then $\Hess_u (0)$   has $k+1$ eigenvalues equal to $- \frac{1}{k}$ and $n-k$ zero eigenvalues.  In this case, (B) holds at $0$ for any unit vector $v$ in the $- \frac{1}{k}$ 
eigenspace.
\end{proof}

\vskip2mm
We get the following immediate consequence of Proposition \ref{p:oneA} and Lemma \ref{l:spherecyl}:

\begin{Cor}
The  viscosity solution is a classical solution.
\end{Cor}

\begin{proof}
By Corollary \ref{c:uc1},
the only thing to check is on the critical set.  However, Proposition \ref{p:oneA} gives   the Hessian  there and this satisfies (B) by
Lemma \ref{l:spherecyl}.
\end{proof}

This completes the proof of Theorem \ref{t:mainarrival}.

\section{Vanishing order along the axis}

The next theorem shows that  $\nabla u$ vanishes faster than linearly in the direction of the axis for a cylindrical singularity $\SS^k \times \RR^{n-k}$ (the $\RR^{n-k}$ factor is the ``axis'').

\begin{Thm}	\label{t:axidi}
If $u$ has a cylindrical singularity at $0$ and $v$ is any unit vector in the direction of the axis, then for some $\alpha>0$
\begin{align}
	\lim_{\delta  \to 0} \, \, \frac{ ( \log |\log \delta| )^{\alpha} \, |\nabla u (\delta \,v)|}{\delta }  = 0 \, .
\end{align}
\end{Thm}

\begin{proof}
Given a vector $w$, let $w_+$ be the projection onto $\RR^{k+1}$ and $w_-$ be the projection onto $\RR^{n-k}$, so  $v_+ = 0$ since $v$ points in the direction of the axis.

Let $\Sigma_s$ for $s \geq 0$ be the rescaled MCF  associated to the blowup at $0$ given by
 \begin{align}	\label{e:rescf}
 	\Sigma_s \equiv  \frac{1}{\sqrt{  -t} } \, \{ u = t \}  \,  {\text{ where }} t = - \e^{-s}  \, .
\end{align}
The   uniqueness of \cite{CM2}{\footnote{See theorem $0.2$ and footnote $6$ in \cite{CM2}; the precise rate follows from theorem $6.9$ and lemma $6.1$.}} gives
$C$   so that  $\Sigma_s$
is a graph over $\SS^k_{\sqrt{2k}} \times \RR^{n-k}$ in the ball $B_{\frac{ \sqrt{ \log s } }{C}}$ with radius $\frac{ \sqrt{ \log s } }{C}$ and center $0$  of a function with $C^1$ norm at most $C/  \sqrt{ \log s}$ for $s$ large.  Using \eqr{e:rescf} to translate this to the original flow, we get graphical control on the $M_t$'s when 
\begin{align} 	\label{e:controls}
	|x|^2 \leq - C' \, t \, \log |\log -t|  \, .
\end{align}

Suppose that $\delta_i$ is a sequence going to $0$.  Given any $\beta \in (0,1)$,  the convergence in the growing scale-invariant region \eqr{e:controls} gives points $y_i$ (for $i$ large) in the graphs   with
\begin{align}	\label{e:findyis}
	(y_i - \delta_i v)_{-} = 0 {\text{ and }} |u(y_i)| \leq \frac{ \delta_i^2}{( \log |\log \delta_i|)^{\beta}} \, .
\end{align}
Since we are in the graphical region, we have 
\begin{align}	\label{e:findyis2}
	|(y_i)_+| \leq \sqrt{-u(y_i)} (\sqrt{2k} + 1) \leq (\sqrt{2k} + 1) \,
	\frac{ \delta_i}{( \log |\log \delta_i|)^{\beta/2}}  \,  .
\end{align}
If $u_c$ was the arrival time for the cylinder itself, 
 we would have $\nabla u_c (y_i)=  - \frac{(y_i)_+}{k}$.  Since $y_i$ is in the graphical region, \eqr{e:findyis2} implies that
 \begin{align}
 	\frac{\left| \nabla u (y_i) \right|}{\delta_i} \leq \frac{ C}{( \log |\log \delta_i|)^{\beta/2}} \, .
 \end{align}
 The theorem follows (for any $\alpha < \beta/2$) from this and using
 the $C^{1,1}$ bound on $u$ to get
\begin{align}
	 |\nabla u (y_i) - \nabla u (\delta_i \, v)|   \leq C\,  |y_i - \delta_i v| = C \, |(y_i)_+|  \, .
\end{align}

\end{proof}

\end{document}